\documentclass[journal,twoside,web]{ieeecolor}

\usepackage{generic}
\usepackage{amsmath,amssymb,bbm}
\let\labelindent\relax
\usepackage{enumitem}
\usepackage{cancel}
\usepackage{graphicx}
\usepackage{balance}
\DeclareMathAlphabet{\mathbbold}{U}{bbold}{m}{n}

\usepackage{multicol}

\usepackage{algorithm}
\usepackage{algpseudocode}
\algrenewcommand\algorithmicrequire{\textbf{Input:}}
\algrenewcommand\algorithmicensure{\textbf{Output:}}

\renewcommand\algorithmicdo{}

\newcommand\NoDo{\renewcommand\algorithmicdo{}}

\newcommand{\rea}{\mathbb{R}}

\newcommand{\reap}{\mathbb{R}_+}
\newcommand{\nat}{\mathbb{N}}

\renewcommand{\int}{\mathbb{I}}

\newcommand{\bone}{\mathbbold{1}}
\newcommand{\bzero}{\mathbbold{0}}
\renewcommand{\int}{\mathbb{Z}}

\newcommand{\G}{\mathcal{G}}

\newcommand{\V}{\mathcal{V}}
\newcommand{\E}{\mathcal{E}}
\newcommand{\N}{\mathcal{N}}
\newcommand{\X}{\mathcal{X}}

\newcommand{\Is}{\mathsf{Id}}

\newcommand{\Ts}{\mathsf{T}}
\newcommand{\Us}{\mathsf{U}}
\newcommand{\Ls}{\mathsf{L}}

\newcommand{\Fs}{\mathsf{F}}
\newcommand{\Gs}{\mathsf{G}}
\newcommand{\wTs}{\widetilde{\Ts}}
\newcommand{\weTs}{\widetilde{\Ts}^e}

\DeclareMathOperator*{\argmin}{\arg\!\min}
\DeclareMathOperator*{\arginf}{\arg\!\inf}

\renewcommand{\bar}[1]{\overline{#1}}
\newcommand{\ubar}[1]{\underline{#1}}

\newcommand{\pmin}{\underline{p}}
\newcommand{\pmax}{\bar{p}}

\DeclareMathOperator{\prox}{prox}
\DeclareMathOperator{\proj}{proj}

\DeclareMathOperator{\fix}{fix}

\definecolor{myred}{rgb}{0.8, 0.0, 0.0}
\definecolor{mygray}{rgb}{0.7, 0.7, 0.7}
\definecolor{orange}{rgb}{1, 0.65, 0.0}

\newcommand{\ev}[1]{\mathbb{E}\left[#1\right]}
\newcommand{\evc}[2]{\mathbb{E}_{#1}\left[#2\right]}

\newcommand{\norm}[1]{{\left\vert\kern-0.25ex\left\vert #1\right\vert\kern-0.25ex\right\vert}}
\newcommand{\mnorm}[1]{{\left\vert\kern-0.3ex\left\vert\kern-0.3ex\left\vert#1\right\vert\kern-0.3ex\right\vert\kern-0.3ex\right\vert}}
\newcommand{\abs}[1]{{\left\vert #1 \right\vert}}

\usepackage{xspace}

\newcommand{\virg}[1]{``#1"}

\usepackage{ntheorem,lipsum}
\theoremheaderfont{\hspace*{\parindent}\bf}
\theoremseparator{.}
\newtheorem{assum}{Assumption}
\newtheorem{thm}{Theorem}
\newtheorem{rem}{Remark}
\newtheorem{prop}{Proposition}
\newtheorem{cor}{Corollary}

\newtheorem{defn}{Definition}

\usepackage[shortcuts]{extdash}
\begin{document}

\title{\fontsize{22pt}{36pt}\selectfont Robust Online Learning over Networks}
\author{Nicola Bastianello$^\star$, \IEEEmembership{Member, IEEE}, Diego Deplano$^\star$, \IEEEmembership{Member, IEEE}, \\
Mauro Franceschelli, \IEEEmembership{Senior, IEEE}, and Karl H. Johansson, \IEEEmembership{Fellow, IEEE}
\thanks{$^\star$ Nicola Bastianello and Diego Deplano are co-first authors as they contributed equally. Nicola Bastianello is the corresponding author.}
\thanks{The work of N. Bastianello and K. H. Johansson was partially supported by the European Union’s Horizon 2020 research and innovation programme under grant agreement No. 101070162, and partially by the Swedish Research Council Distinguished Professor Grant 2017-01078 Knut and Alice Wallenberg Foundation Wallenberg Scholar Grant.}
\thanks{The work of D. Deplano was supported by the project e.INS- Ecosystem of Innovation for Next Generation Sardinia (cod. ECS 00000038) funded by the Italian Ministry for Research and Education (MUR) under the National Recovery and Resilience Plan (NRRP) - MISSION 4 COMPONENT 2, \virg{From research to business} INVESTMENT 1.5, \virg{Creation and strengthening of Ecosystems of innovation} and construction of \virg{Territorial R\& D Leaders}.}
\thanks{N. Bastianello and K. H. Johansson are with the School of Electrical Engineering and Computer Science and Digital Futures, KTH Royal Institute of Technology, Stockholm, Sweden.
Emails: {\tt \{nicolba,kallej\}@kth.se}}
\thanks{D. Deplano and M. Franceschelli are with DIEE, University of Cagliari, 09123 Cagliari, Italy.
Emails: {\tt \{diego.deplano,mauro.franceschelli\}@unica.it}}
}

\maketitle
\begin{abstract}
The recent deployment of multi-agent networks has enabled the distributed solution of learning problems, where agents cooperate to train a global model without sharing their local, private data.
This work specifically targets some prevalent challenges inherent to distributed learning: (i) online training, i.e., the local data change over time; (ii) asynchronous agent computations; (iii) unreliable and limited communications; and (iv) inexact local computations.
To tackle these challenges, we apply the Distributed Operator Theoretical (DOT) version of the Alternating Direction Method of Multipliers (ADMM), which we call \virg{DOT-ADMM}.
We prove that if the DOT-ADMM operator is metric subregular, then it converges with a linear rate for a large class of (not necessarily strongly) convex learning problems toward a bounded neighborhood of the optimal time-varying solution, and characterize how such neighborhood depends on~$\text{(i)--(iv)}$.
We first derive an easy-to-verify condition for ensuring the metric subregularity of an operator, followed by tutorial examples on linear and logistic regression problems.
We corroborate the theoretical analysis with numerical simulations comparing DOT-ADMM with other state-of-the-art algorithms, showing that only the proposed algorithm exhibits robustness to (i)--(iv).
\end{abstract}


\section{Introduction}\label{sec:introduction}
In recent years, significant technological advancements have enabled the deployment of multi-agent systems across various domains, including robotics, power grids, and traffic networks~\cite{molzahn_survey_2017,nedic_distributed_2018}.
These systems consist of interconnected agents that leverage their computational and communication capabilities to collaborate in performing assigned tasks.
Many of these tasks -- such as estimation~\cite{Montijano21,Deplano23}, coordination and control~\cite{Deplano20,Santilli21,Deplano23novel}, resilient operation~\cite{Shang20,Sheng22,santilli2022secure}, and learning~\cite{boyd2011distributed,qian_distributed_2022,park_communication-efficient_2021} -- {can be formulated as \textit{distributed optimization problems}, see~\cite{nedic_distributed_2018,notarstefano_distributed_2019,yang_survey_2019,li2023survey} for some recent surveys.}
In this context, this work focuses specifically on distributed optimization algorithms for learning under network constraints.

Traditional machine learning methods require transmitting all the data collected by the agents to a single location, where they are processed to train a model.
However, communicating raw data exposes agents to privacy breaches, which is inadmissible in many applications such as healthcare and industry~\cite{gafni_federated_2022}.
Moreover, it is often the case that the agents collect new data over time, which requires another round of data transmission and, as a result, both an increased privacy vulnerability and the need of repeating the centralized training task.
In the alternative distributed approach, the agents within the network first process their data to compute an approximate local model, and then refine such model by sharing it with their peers and agreeing upon a common model that better fits all the distributed data sets, with the goal of improving the overall accuracy. This strategy, however, poses some practical challenges -- discussed in the next section -- both at the computation level, such as computing with different speed and precision, and at the communication level, such as loss and corruption of packets.

The focus of this paper is thus on solving online learning problems in a distributed way while addressing these practical challenges -- discussed in detail in Section~\ref{subsec:challenges} -- by applying a distributed operator theoretical (DOT) version of the alternating direction method of multipliers (ADMM), which we call \virg{DOT-ADMM}. In principle, other different approaches developed in the distributed optimization literature can be applied to this problem, which are reviewed in Section~\ref{sec:review}.
Section~\ref{sec:maincont} outlines the main technical contributions of the manuscript, regarding the linear convergence of DOT-ADMM for (not necessarily strongly) convex problems based on novel theoretical results on stochastic and metric subregular operators.

\subsection{Practical challenges in online learning over networks}\label{subsec:challenges}

\subsubsection{Asynchronous agents computations}
The agents cooperating for the solution of a learning problem are oftentimes highly heterogeneous, in particular in terms of their computational resources~\cite{gafni_federated_2022}; consequently, the agents may perform local computations at different rates.
The simple solution of synchronizing the network by enforcing that all agents terminate the $k$-th local computation at iteration $k$ entails that better-performing agents must wait for the slower ones (cf.~\cite[Fig.~2]{peng_coordinate_2016}).
Therefore, in this paper we allow the agents to perform local processing at their own pace, which is a more efficient strategy than enforcing synchronization.

\subsubsection{Unreliable communications}
In real-world scenarios, the agents have at their disposal imperfect channels to deliver the locally processed models to their peers.
One problem that may occur, particularly when relying on wireless communications, is that transmissions from one agent to another may be lost in transit (e.g., due to interference~\cite{qian_distributed_2022}). When a transmission is lost, the newly processed local model of an agent is not delivered to its neighboring agents, but the algorithm needs to be robust to this occurrence.
By \virg{robust}, we refer to the ability of the algorithm to mitigate the effect of the packet-losses without taking any specific action. Indeed, the design of specific procedures to address packet-losses may require some additional knowledge at the transmitter level, which is unfeasible in some applications.
A second problem that must be faced, especially when the local models stored by the agents are high dimensional, is the impracticability of sharing the exact local model over limited channels (e.g., when training a deep neural network). To satisfy limited communications constraints, different approaches have been explored, foremost of which is \textit{quantization/compression} of the messages exchanged by the agents~\cite{park_communication-efficient_2021}. 
But quantizing a communication implies that an inexact version of the local models is shared by the agents, which can be seen as a disturbance in the communication.
Therefore, we consider both packet-loss and packet-corruption during the communications between agents, which allow us to deal with the two above-mentioned problems.

\subsubsection{Inexact local computations} 
In learning applications, the local training performed by the agents may depend on large data-sets, and thus be computationally demanding. This is especially relevant in online set-ups, where training needs to be completed within the interval of time $[k, k+1)$.
To solve this issue, so-called \textit{stochastic gradients} are employed, which construct an approximation of the local gradients using a limited number of data points~\cite{koloskova_unified_2020}. This implies that every time a stochastic gradient is applied by an agent, some error is introduced in the algorithm.
The algorithm we propose in this manuscript needs to compute the proximal of the local cost function, where the proximal operator finds the point that minimizes a function while also being close to its argument.
However, unless the cost is proximable~\cite{Parikh14} and there is a closed form, the proximal needs to be computed via an iterative scheme, such as gradient descent. But again due to the limited computational time $[k, k+1)$ allowed to the agent, the proximal can be computed only with a limited number of iterations, introducing an approximation. The issue may be further compounded by the use of stochastic gradients instead of full gradients.

\subsection{Review of the state-of-the-art}\label{sec:review}
Many algorithms in the state-of-the-art for solving optimization and learning problems are built upon (sub)gradient methods~\cite{nedic_distributed_2018}.
Despite their effectiveness, these methods are limited to achieving, at best, sub-linear convergence, necessitating the adoption of diminishing step-sizes, even in the case of strongly convex problems.
Another class of suitable algorithms is that of gradient tracking, which can achieve convergence with the use of fixed step-sizes~\cite{xin_general_2020}. On the one hand, different gradient tracking methods have been proposed for application in an online learning context, see \textit{e.g.}~\cite{yuan_can_2020,carnevale_gtadam_2023}. On the other, the use of robust average consensus techniques has also enabled the deployment of gradient tracking for learning under the constraints of Section~\ref{subsec:challenges}, see \textit{e.g.}~\cite{xu_convergence_2018,bof_multiagent_2019,tian_achieving_2020,li_variance_2022,lei_distributed_2022}.
However, gradient tracking algorithms may suffer from a lack of robustness to some of these challenges~\cite{bin_stability_2022}.

Our approach falls in a different branch of research, which is based on the alternating direction method of multipliers (ADMM). ADMM-based algorithms have turned out to be reliable and versatile for distributed optimization~\cite{boyd2011distributed,peng_arock_2016}.
In particular, ADMM has been shown to be robust to asynchronous computations~\cite{wei_o1k_2013,chang_asynchronous_2016,peng_arock_2016}, packet losses~\cite{majzoobi_analysis_2018}, and both~\cite{Bastianello21}. Additionally, its convergence with inexact communications has been analyzed in~\cite{majzoobi_analysis_2019}, and the impact of inexact local computations has been studied in~\cite{xie_siadmm_2020}.
Moreover, the convergence of ADMM-based algorithms under network constraints has been usually shown to occur at a sub-linear rate, whereas a linear rate can be proved only under additional assumptions, such as strong convexity~\cite{Bastianello21}).

Instead, the algorithm we propose is shown to converge with a linear rate without the strong convexity assumption, but under a milder set of assumptions, while facing at the same time all the challenges described in Section~\ref{subsec:challenges}.
\subsection{Main contributions}\label{sec:maincont}
\noindent DOT-ADMM has the following set of features:
\begin{itemize}
    \item Convergence with a linear rate for a wide class of learning problems (e.g., linear and logistic regression problems);
    \item Applicability in an online scenario where the data sets available to the agents change over time;
    \item Robustness to asynchronous and inexact computations of the agents;
    \item Robustness to faulty and noisy communications.
\end{itemize}
The main theoretical results of the paper are as follows:
\begin{itemize}
    \item Time-varying stochastic operators that are averaged and metric subregular operators converge linearly (in mean) and almost surely to a neighborhood of the time-varying set of fixed points without assuming that there exists a common fixed point (see Theorem~\ref{thm:glob-msr}). 
    \item DOT-ADMM is proved to converge for a large class of online learning problems with (not necessarily strongly) convex local costs under the challenging scenario described in Section \ref{subsec:challenges} by relying on the metric subregularity property of the DOT-ADMM operator (see Theorem~\ref{thm:stochastic+noise+varying}).
    Complementary results are provided for simpler scenarios where some of the challenges do not come into play (see Corollary~\ref{cor:stochastic}) and for the case in which metric subregularity holds in a subset of the state space (see Theorem~\ref{thm:local}).
    \item An easy-to-verify sufficient condition to ensure metric subregularity of an operator is provided in Proposition \ref{prop:ULaffMS}, which basically requires the operator to be bounded by two affine operators.
    This result is then applied to standard learning problems, such as linear regression (see Proposition~\ref{prop:linear}), robust linear regression (see Proposition~\ref{prop:linear}), and logistic regression (see Proposition~\ref{prop:logistic}).
    These results can be also used as tutorial examples for other learning problems with different regression models.
    \item Extensive numerical simulations of DOT-ADMM together with a comparison with other state-of-the-art algorithms are provided in Section \ref{sec:numerical}, revealing the outperforming performance of DOT-ADMM in terms of convergence time and resilience to the challenging scenario considered in the manuscript.
\end{itemize}

\subsection{Outline}
Section~\ref{sec:notation} provides the notation used throughout the paper, gives useful preliminaries on graph theory and operator theory, and formalizes the online optimization problem of interest together with some technical working assumptions.
In Section~\ref{sec:algorithm} we formalize the technical challenges usually faced when solving online learning problems, and we present $\text{DOT-ADMM}$ as a suitable protocol to be implemented in networks to solve these problems in a distributed manner while facing all the challenges.
Section~\ref{sec:convergence} is devoted to the proof of the convergence result anticipated in Section~\ref{sec:algorithm}, which relies on a novel foundational result on stochastic operators that are metric subregular.
Section~\ref{sec:distributed-application} outlines how the proposed algorithm can be applied to different learning problems, with a focus on linear and logistic regression problems.
In Section~\ref{sec:numerical} several numerical results are carried out, and Section~\ref{sec:conclusion} gives some concluding remarks.

\section{Preliminaries and Problem formulation}\label{sec:notation}
The set of real and integer numbers are denoted by $\rea$ and $\int$, respectively, while $\reap$ and $\nat$ denote their restriction to positive entries.
Matrices are denoted by uppercase letters, vectors and scalars by lowercase letters, while sets and spaces are denoted by uppercase calligraphic letters.
The identity matrix is denoted by $I_n$, $n\in\nat$, while the vectors of ones and zeros are denoted by $\bone_n$ and $\bzero_n$; subscripts are omitted if clear from the context.
Maximum and minimum of an $n$-element vector ${u=[u_1,\ldots, u_n]^\top}$, are denoted~by $\bar{u}=\max_{i =1,\ldots,n} u_{i}$ and $ \ubar{u}=\min_{i =1,\ldots,n} u_{i}$, respectively.

\subsection{Preliminaries}

\subsubsection{Networks and graphs}
We consider networks modeled by undirected graphs ${\G=(\V,\E)}$, where $\V=\{1,\ldots,n\}$, $n\in\nat$, is the set of \emph{nodes}, and $\E\subseteq \V\times \V$ is the set of \emph{edges} connecting the nodes.
An undirected graph $\mathcal{G}$ is said to be \emph{connected} if there exists a sequence of consecutive edges between any pair of nodes $i,j \in \V$.
Nodes $i$ and $j$ are \emph{neighbors} if there exists an edge ${(i,j)\in \E}$.
The set of neighbors of node $i$ is denoted by $\mathcal{N}_i=\left\{j\in \V: (i,j)\in \E\right\}$.
For the sake of simplicity, we consider graphs with self-loops, i.e., $i\in \mathcal{N}_i$, and denote by $\eta_i = |\N_i|$ the number of neighbors and by $\xi=2|\E|=\eta_1+\cdots+\eta_n$ twice the number of undirected edges in the network.

\subsubsection{Operator theory}
We introduce some key notions from operator theory in finite-dimensional Euclidean spaces, i.e., vector spaces $\rea^n$ with $\norm{\cdot}$ and distance $d$
$$
\norm{ x}=\sqrt{ x^\top  x},\qquad d( x, y) = \norm{ x- y}.
$$
Operators $\Fs:\rea^n\rightarrow \rea^n$ are denoted with block capital letters. An operator is \emph{affine} if there exist a matrix $A\in\rea^{n\times n}$ and a vector $b\in\rea^n$ such that ${\Fs :  x\mapsto A x+b}$, and \emph{linear} if $b=\bzero$. The linear operator associated to the identity matrix $I$ is defined by ${\Is: x \mapsto I x}$.
Given $\Fs:\rea^n\rightarrow\rea^n$, ${\fix (\Fs)=\{ x\in\rea^n:\Fs( x)= x\}}$ denotes its set of \emph{fixed points}.
By further defining the \emph{projection operator} of a point $x$ over a non-empty set $\X$ as
$$
\proj_{\X}(x) = \arginf_{y\in\X} \norm{x-y},
$$
the distance of point $x$ from the set $\X$ is denoted by
$$
d_{\X}( x) = \norm{ x -  \proj_{\X}(x)} 
$$
When $\X$ is the set of fixed points of an operator $\Fs$, we use the shorthand notations $d_{\Fs}:=d_{\fix(\Fs)}$ and $\proj_{\Fs}:=\proj_{\fix(\Fs)}.$
With this notation, we now define some properties of operators, which are pivotal in this paper.

\begin{defn}\label{def:metsubreg}
An operator $\Fs : \rea^n \to \rea^n$ is metric subregular if there is a positive constant  $\gamma>0$ such that
\begin{equation}\label{eq:metric-subregularity}
    d_{\Fs}( x) \leq \gamma \norm{(\Is - \Fs)  x},\qquad \forall x\in\rea^n.
\end{equation}
When $\gamma$ is known, $\Fs$ is said to be $\gamma$-metric subregular.
\end{defn}
\begin{defn}
An operator $\Fs : \rea^n \to \rea^n$ is nonexpansive if $$\norm{\Fs( x)-\Fs( y)}\leq \norm{ x- y},\qquad \forall x,y\in\rea^n.$$
Moreover, the operator $\Gs:=(1-\alpha)\Is+\alpha\Fs$ with $\alpha\in(0,1)$ is $\alpha$-averaged if $\Fs$ is nonexpansive, or, equivalently, if
\begingroup
\medmuskip=-1mu
\thinmuskip=-1mu
\thickmuskip=-1mu
$$
\norm{\Gs( x)-\Gs( y)}^2\leq \norm{ x- y}^2 
- \frac{1-\alpha}{\alpha}\norm{(\Is-\Gs)( x)-(\Is-\Gs)( y)}^2.
$$
\endgroup
\end{defn}

A function $f : \X\rightarrow [-\infty,+\infty]$ with $\X\subseteq \rea^n$ is: \textit{proper} if its domain $\{x\in\X\ | \ f(x)<+\infty\} $ is not empty and $-\infty \not\in f(\X)$~\cite[Definition~1.4]{Bauschke2017}; \textit{lower semicontinuous} if its epigraph $\{ (x, t) \in \X \times \rea \ | \ f(x) \leq t \}$  with $\X\subseteq \rea^n$ is closed in $\rea^n \times \rea$~\cite[Lemma~1.24]{Bauschke2017}; \textit{convex} if its epigraph is a convex subset of $\X\times \rea$~\cite[Definition~8.1]{Bauschke2017}.
Let $\Gamma_0^n$ be the set of proper, lower semicontinuous, convex functions $f$ from ${\X\subseteq \rea^n}$ to $\rea \cup \{ +\infty \}$.
Then, the proximal operator of $f\in \Gamma_0^n$ is defined by
$$
\prox^{\rho}_{f}( y)=\argmin_{ x}\left\{f( x)+\frac{1}{2\rho}\norm{ x-
 y}^2\right\},
$$
where $\rho>0$ is a penalty parameter; if $\rho = 1$, it is omitted. 
We note that the projection operator is a particular case of the proximal operator when applied to the indicator function $\iota:\rea^n\rightarrow\rea^n\cup\{+\infty\}$ defined as $\iota_{\X}(x) = 0$ if $x \in \X$, and $\iota_{\X}(x)=+\infty$ otherwise.

\subsection{Problem formulation}
We consider a network of $n$ agents linked according to an undirected, connected graph ${\mathcal{G}=(\V,\E)}$. Each agent $i\in\V$ has a vector state ${ x_i(k) \in \rea^p}$ with $p\in\nat$, and has access to a \emph{time-varying} local cost $f_{i,k} : \rea^p \to \rea \cup \{ +\infty \}$, $k \in \nat$. The objective of the network is to solve the optimization problem $\min_x \sum_{i\in\V} f_{i,k}(x)$, which can be reformulated in the following distributed form:
\begin{equation}\label{eq:online-distributed-optimization}
\begin{aligned}
	 &\min_{ x_i} \sum_{i \in\V} f_{i,k}( x_i) \\
	&\:\:\text{s.t.}  \quad x_i =  x_j \ \text{if} \ (i,j) \in \E
\end{aligned}
\end{equation}
whose set of solutions is denoted by $\X^\star_k$. We make the following two standing assumptions, which are standard assumptions in online optimization~\cite{dallanese_optimization_2020}.
\begin{assum}\label{as:convexity}
At each time $k\in\nat $, the local cost functions $f_{i,k}$ of problem~\eqref{eq:online-distributed-optimization} are proper, lower semi-continuous, and convex, i.e., $f_{i,k}\in \Gamma_0^p$.
\end{assum}
\begin{assum}\label{as:time-variability}
The set of solutions $\X^\star_k$ to problem~\eqref{eq:online-distributed-optimization} is non-empty at each time $k\in\nat$. In addition, the minimum distance between two solutions at consecutive times is upper-bounded by a nonnegative constant $\sigma \geq 0$, i.e.,
$$
    \sup_{k \in \nat} \,\, \inf_{x^\star_k \in \X^\star_{k}} d_{\X^\star_{k-1}}(x^\star_k) \leq \sigma.
$$

\end{assum}

\begin{rem}
If one only considers Assumption~\ref{as:convexity}, the set of solutions $\X_k^\star$ may be empty \cite[Proposition~11.15]{Bauschke2017}; thus, Assumption~\ref{as:time-variability} requires that $\X_k^\star\neq \emptyset$.
On the other hand, the set of solutions $\X_k^\star$ may also contain infinitely many solutions and be unbounded; thus, Assumption~\ref{as:time-variability} requires that there exists an upper bound $\sigma \geq 0$ (that holds uniformly over time) to the distance between any solution $x^\star_k \in \X_k^\star$ and its projection onto the set of solutions at the previous step, i.e., $\X_{k-1}^\star$. 
\end{rem}

The dynamic nature of the problem implies that the solution -- in general -- cannot be reached exactly, but rather that the agents' states will reach a neighborhood of it~\cite{simonetto_time-varying_2020,dallanese_optimization_2020}.
Our goal is thus to quantify how closely the agents can track the optimal solution over networks characterized by the following challenging conditions, as discussed in Section~\ref{subsec:challenges}:
\begin{itemize}
    \item asynchronous agents computations;
    \item unreliable communications;
    \item inexact local computations.
\end{itemize}
The next section introduces the proposed algorithm along with the formal description of the above challenges and the main working assumptions.

\section{Proposed algorithm and \\convergence results}\label{sec:algorithm}
To solve the problem in eq. \eqref{eq:online-distributed-optimization}, we employ the Distributed Operator Theoretical (DOT) version of the Alternating Direction Method of Multipliers (ADMM), which we call \virg{$\text{DOT-ADMM}$}. It is derived by applying the relaxed Peaceman-Rachford splitting method to the dual of problem in eq.~\eqref{eq:online-distributed-optimization} (see~\cite{Bastianello21} and references therein), and its distributed implementation is detailed in Algorithm~\ref{alg:distributed-admm}.

Each agent $i \in \V$ first updates its local state $x_i \in \rea^p$ according to eq.~\eqref{eq:admm-x}, then sends some information to each neighbor $j\in\N_i$ within the packet $y_{i\to j}$. The agents are assumed to be heterogeneous in their computation capabilities, therefore at each time step only some of the agents are ready for the communication phase. In turn, after receiving the information from its neighbors, each agent updates its auxiliary state variables $z_{ij}\in\rea^p$ with $j\in\N_i$ according to eq.~\eqref{eq:admm-z}.
Note that the local update in eq.~\eqref{eq:admm-x} depends only on information available to agent $i$, while for the auxiliary update in eq.~\eqref{eq:admm-z} the agent needs to first receive the aggregate information $y_{j\to i}(k)$ from the neighbor $j$.

\begin{algorithm}[!t]
\caption{\small Distributed Operator Theoretical (DOT) ADMM}
\label{alg:distributed-admm}
\begin{algorithmic}
\Require{For each agent $i \in \V$ initialize the auxiliary variables $\{ z_{ij}(0) \}_{j \in \N_i}$; choose the relaxation $\alpha \in (0, 1)$ and the penalty $\rho > 0$.}
\Ensure{Each agent returns $x_i(k)$ that is an (approximated) solution to the distributed optimization problem in eq. \eqref{eq:online-distributed-optimization}.}
\NoDo
\State \hspace{-1.3em} \textbf{for $k=1,2,\ldots$ each active agent $i\in\V$}
\State {\color{mygray}// asynchronous computations and inexact updates} 
\State receives a local cost $f_{i,k}$ and applies the local update 
\begingroup
\medmuskip=1mu
\thinmuskip=1mu
\thickmuskip=1mu
\small
\begin{equation}\label{eq:admm-x}
x_i(k) =  \Fs_{i,k}(z(k-1)) := \prox_{f_{i,k}}^{1/\rho \eta_i} \left( \frac{1}{\rho \eta_i} \sum_{j \in \mathcal{N}_i}  z_{ij}(k-1) \right)\vspace{-1em}
\end{equation}
\normalsize
\endgroup

\For{\textbf{each agent} $j \in \N_i$}
\State{transmits the packet
\small
$$
     y_{i \to j}(k) =  2\rho  x_i(k) -  z_{ij}(k-1)
$$
\normalsize}
\EndFor

\For{\textbf{each packet $ y_{j \to i}$ received by agent $j \in \N_i$}}
\State {\color{mygray}// noisy communications with packet loss} 
\State updates the auxiliary variable
\small
\begin{equation}\label{eq:admm-z}
     z_{ij}(k) = \Ts_{ij,k}(z(k-1)) := (1 - \alpha)  z_{ij}(k-1) + \alpha  y_{j\to i}(k)\vspace{-1em}
\end{equation}
\normalsize
\EndFor\\
\hspace{-1em}\textbf{end for}
\end{algorithmic}
\end{algorithm}

Algorithm \ref{alg:distributed-admm} also makes explicit where the sources of stochasticity discussed in Section~\ref{subsec:challenges} come into play: asynchronous agents' computations and packet-loss prevent the updates in eqs. \eqref{eq:admm-x}-\eqref{eq:admm-z} to be performed at each time $k\in\nat$, whereas inexact local computations and noisy communications make these updates inaccurate.
We model these sources of stochasticity by means of the following random variables:
\begin{itemize}
    \item ${\beta_{ij}(k)\sim \text{Ber}(p_{ij})}$ are Bernoulli random variables\footnote{i.e., $\beta_{ij}(k)=1$ with probability $p_{ij}$, ${\beta_{ij}(k)=0}$ with probability $1-p_{ij}$.} with $p_{ij}\in(0,1]$ modeling the asynchronous agents' computations and packet-loss: $p_{ij}$ denotes the probability that agent $j$ has completed its computation and packet $y_{j\to i}$ successfully arrives to agent $i$;
    \item $u_{i}(k)$ and $v_{ij}(k)$ are i.i.d. random variables which represent, respectively, the additive error modeling inexact local updates of $x_i$ and noisy transmission of $ y_{j\to i}(k)$ sent by node $j$ to node $i$.
\end{itemize}

With these definitions in place, we define the perturbed variables as follows
$$
\begin{aligned}
     \widetilde{x}_i(k) &= x_i(k) +u_i(k)\\
     \widetilde{z}_{ij}(k) &= z_{ij}(k)  + \alpha v_{ij}(k)
\end{aligned}
$$
One can further notice that the additive error $u_i(k)$ on $x_i(k)$ is also an additive error on $y_{j\to i}$ (scaled by a factor equal to $2\rho$). Therefore, one can consider only one source of error $e_{ij}=v_{ij}(k) + 2\rho u_i(k)$ and write the perturbed updates as
\begingroup
\medmuskip=3mu
\thinmuskip=3mu
\thickmuskip=3mu
\begin{subequations}\label{eq:admm-complete}
\begin{align}
     x_i(k) &= \Fs_{i,k}( \widetilde{z}(k-1))\label{eq:admm-compact-x}\\
     \widetilde{z}_{ij}(k) &= \begin{cases}
    \Ts_{ij,k}(\widetilde{z}(k-1)) + \alpha  e_{ij}(k) & \text{if } \beta_{ij}(k) = 1 \\
     \widetilde{z}_{ij}(k-1)& \text{otherwise}
\end{cases}\label{eq:admm-compact-z}
\end{align}
\end{subequations}
\endgroup
where the operators $\Fs_{i,k}$, $\Ts_{ij,k}$ are those of eqs. \eqref{eq:admm-x}-\eqref{eq:admm-z}.

With this notation, we formalize next the challenging assumptions under which the problem in eq. \eqref{eq:online-distributed-optimization} must be solved.

\begin{assum}\label{as:random-updates}
At each time step, each node $j\in\V$ has completed its local computation and successfully transmits data to its neighbors $i\in\N_j$ with probability $p_{ij} \in (0, 1]$.
\end{assum}
The minimum and maximum among the probabilities of Assumption \ref{as:random-updates} are denoted by
\begin{equation}\label{eq:probs}
    \pmin = \min_{(i,j) \in \E} p_{ij},\qquad \pmax = \max_{(i,j) \in \E} p_{ij}.    
\end{equation}
\begin{assum}\label{as:additive-error}
Each node $i\in \V$ updates its local variable $x_i(k)$ with an additive error ${u_{i}(k)\in\rea^p}$ and receives information from neighbor $j\in\mathcal{N}_i$ with an additive noise ${v_{ij}(k)\in\rea^p}$ such that the overall error ${e_{ij}(k) = v_{ij}(k)+2\rho u_i(k)}$ on the update of the auxiliary variable $z_{ij}(k)$ is bounded by ${\ev{\norm{e_{ij}(k)}}\leq \nu_e<\infty}$.
\end{assum}

\subsection{Convergence results}\label{subsec:algorithm-convergence}
For the convenience of the reader, we state next our main results, while we postpone their proofs to Section~\ref{sec:convergence}.
We begin with the following theorem, which characterizes the mean linear convergence of DOT-ADMM in the stochastic scenario described in Section~\ref{subsec:challenges} and formalized in Section~\ref{sec:algorithm}.

\begin{thm}[Linear convergence]
\label{thm:stochastic+noise+varying}
Consider the online distributed optimization in problem~\eqref{eq:online-distributed-optimization} under Assumptions~{$\text{\ref{as:convexity}-\ref{as:time-variability}}$}, and a connected network of agents that solves it by running DOT-ADMM under Assumptions~\ref{as:random-updates}-\ref{as:additive-error}.
If the DOT-ADMM operator $\Ts_k$, defined block-wise by $\Ts_{ij,k}$ as in eq.~\eqref{eq:admm-z}, is $\gamma$-metric subregular, then:
\begin{itemize}
    \item[i)] there is an upper bound to the distance between the current solution $x(k)$ and the set of optimal solutions $\X_k^\star$ that holds in mean for all $k \in \nat$, and this bound has a linearly decaying dependence on the initial condition:
    \begingroup
    \medmuskip=2mu
    \thinmuskip=2mu
    \thickmuskip=2mu
    \begin{equation}\label{eq:punctual-bound}
        \hspace{-1em}\ev{d_{\X_k^\star}( x(k))} = O\left( \mu^k d_{\Ts_0}(x(0)) + \frac{1 - \mu^k}{1 - \mu} (\nu_e + \sigma) \right)
    \end{equation}
    \endgroup
    where the rate of convergence $\mu\in(0,1) $ is given in.~\eqref{eq:mu}.
    
    \item[ii)] there is an upper bound for $d_{\X_k^\star}( x(k))$ that holds almost surely when $k \to \infty$, and this bound does not depend on the initial condition:
    \begin{equation}\label{eq:asymptotic-bound}
        \limsup_{k \to +\infty} d_{\X_k^\star}( x(k)) = O\left( \frac{ \nu_e + \sigma}{1 - \mu} \right).
    \end{equation}
\end{itemize}
\end{thm}
The following corollary makes explicit how the results of Theorem~\ref{thm:stochastic+noise+varying} become stronger when some challenges are not considered.
\begin{cor}[Particular cases]
\label{cor:stochastic}
Consider the scenario of Theorem~\ref{thm:stochastic+noise+varying} and the following simplified scenarios:
\begin{itemize}
    \item[(a)] the cost functions $f_{i,k}=f_i$ are static, i.e., $\sigma = 0$;
    \item[(b)] communications are noiseless and computations are exact, i.e., there are no additive errors $\nu_e = 0$;
    \item[(c)] communications are synchronous.
\end{itemize}
Then the results of Theorem~\ref{thm:stochastic+noise+varying} become:
\begin{itemize}
    \item[i)] $(a)$ implies that the distance $d_{\X_k^\star}( x(k))$ converges linearly to $O(\nu_e)$ in mean;
    \item[ii)] $(b)$ implies that the distance $d_{\X_k}( x(k))$ converges linearly to $O(\sigma)$ in mean;
    \item[iii)] $(a) \land (b)$ implies that the distance $d_{\X_k}( x(k))$ converges linearly to zero in mean square with rate $\mu^2$; moreover, $ x(k)$ almost surely converges to the set of solutions $\X^\star$;
    \item[iv)] $(a) \land (b) \land (c)$ implies that $ x(k)$ converges linearly with rate $\mu^2$ and almost surely to the set of solutions~$\X^\star$. 
\end{itemize}
\end{cor}

Building upon Theorem \ref{thm:stochastic+noise+varying}, we also prove that linear convergence holds for strongly convex and smooth costs as the iterative solution approaches a neighborhood of the optimal solutions. This result, which encompasses that of~\cite{Bastianello21}, is termed as \textit{eventual linear convergence}.
\begin{thm}[Eventual linear convergence]
\label{thm:local}
Consider the scenario of Theorem~\ref{thm:stochastic+noise+varying}, when the cost functions $f_{i,k}=f_i$ are static and there are no additive errors $\nu_e=0$.
If the costs are strongly convex and twice continuously differentiable, then there is a finite time $k^\star \in \nat$ such that, for any initial condition:
\begin{itemize}
    \item[i)] (global) for $k \leq k^\star$, the distance $d_{\X_k}( x(k))$ decays sub-linearly in mean square;

    \item[ii)] (local) for $k > k^\star$, the distance $d_{\X_k}( x(k))$ converges linearly to zero in mean square and $x(k)$ almost surely converges to the set of solutions $\X^\star$ for $k \to \infty$.
\end{itemize}
\vspace{-1em}
\end{thm}

\subsection{Discussion of the results}\label{sec:discussion}

\subsubsection{Convergence rate and error bounds}\label{sec:convrate}
The  value of the convergence rate $\mu\in(0,1)$, resulting from the punctual upper bound to the tracking error in eq. \eqref{eq:punctual-bound} provided by Theorem~\ref{thm:stochastic+noise+varying},~is 
\begin{equation}\label{eq:mu}
    \mu = \sqrt{1 - \frac{(1-\alpha)\pmin}{\alpha\lambda} }, \quad \lambda > \max\left\{\gamma^2,\frac{(1-\alpha)\pmin}{\alpha}\right\}.
\end{equation}
where $\pmin\in(0,1]$ is the minimum probability as in eq. \eqref{eq:probs} that any node completes both the local computation and the transmission tasks, $\alpha\in(0,1)$ is the relaxation parameter of DOT-ADMM and $\gamma>0$ is the metric-subregularity constant of the operator $\Ts_k$ ruling the iterations of DOT-ADMM. 
The presence of random updates leads to a worse convergence rate compared to the convergence rate $\mu_s$ attained when all coordinates update at each iteration ($\pmin=1$), indeed, $\mu_s\leq \mu$.
This is in line with the results proved in~\cite{combettes_stochastic_2015}, and makes intuitive sense since less frequent updates (smaller values of $\pmin$) lead to slower convergence (higher values of $\mu$).

On the other hand, it is possible to make the convergence arbitrarily faster by selecting higher values of the relaxation constant $\alpha$, which, however, worsen the asymptotic error bound in eq. \eqref{eq:asymptotic-bound}: therefore, $\alpha$ constitutes a trade-off between the convergence rate and the asymptotic error. 
Another important role is played by the additive noise (through $\nu_e$) and by the time-variability of the costs (through $\sigma$), which prevent DOT-ADMM from converging to the optimal solution by introducing a non-zero term in both the punctual and asymptotic upper bounds: when these non-idealities are not considered, then one recovers exact convergence as pointed out in Corollary \ref{cor:stochastic}.
We also remark that the scaling constants hidden by the $O(\cdot)$ notation depend on the specific structure of the network (through the number of edges $\abs{\E})$.

\subsubsection{Simplified scenarios and mean square convergence}
The results of Corollary~\ref{cor:stochastic} particularize Theorem~\ref{thm:stochastic+noise+varying} for simplified scenarios in which some of the challenges faced in this work are not taken into account.
Removing the time-variability of the costs ($\sigma=0$) implies that the solution provided by DOT-ADMM converges asymptotically within an error from the optimal solution that is bounded by only $\nu_e$; instead, removing the sources of error ($\nu_e=0$) makes the error bounded by only $\sigma$.
When both these challenges are not considered ($\sigma=\nu_e=0$) then exact convergence to the set of optimal solutions can be achieved. Additionally, the linear convergence of DOT-ADMM holds not only in mean but also in mean square. This is a remarkable result since it holds even though the problem is not strongly convex and regardless of the challenging time-varying, unreliable, and asynchronous scenario (see Section \ref{sec:disc_sc}).

\subsubsection{Comparison with \cite{Bastianello21} and strongly convex problems}\label{sec:disc_sc}
The result of Theorem~\ref{thm:local} clarifies the advantage of the metric-subregularity property against strong convexity of the problem.
Indeed, Theorem \ref{thm:stochastic+noise+varying} proves that metric subregularity of the DOT-ADMM operator is sufficient for linear convergence in convex optimization problems. Theorem~\ref{thm:local} proves that strong convexity of the problem (under the additional assumption of twice differentiability of the costs) is sufficient for local linear convergence after a finite time $k^\star$, a behavior that we termed \textit{eventual linear convergence}.
The linear and logistic regression learning problems discussed in Section \ref{sec:distributed-application} constitute examples of problems that are not strongly convex but for which DOT-ADMM converges linearly because the updates are ruled by a metric subregular operator. Additionally, sub-linear convergence can be experienced for some strongly convex problems; an example is given by the scalar regularized costs $f_i(x_i)= \sqrt[3]{x_i^{4}}+\frac{\epsilon}{2}x_i^2$.
Therefore, strong convexity is neither necessary nor sufficient for global linear convergence; instead, metric subregularity of the DOT-ADMM operator is sufficient for convex problems due to Theorem 1.

Theorem \ref{thm:local} also shows how the analysis in this paper subsumes that of~\cite{Bastianello21}.
In particular, in the case of strongly convex and twice differentiable costs, DOT-ADMM converges sub-linearly until a sufficiently small neighborhood of the unique optimal solution $x^\star$ is reached.
Thereafter, the costs are well approximated by a quadratic function around $x^\star$ that, in turn, implies metric subregularity of the DOT-ADMM operator $\Ts_k$, and thus linear convergence follows by Theorem 1, finding as a special case the result of~\cite{Bastianello21}.
Thus, one of the main differences between this paper's contribution w.r.t. \cite{Bastianello21} is that \cite{Bastianello21} requires strong convexity and twice differentiability of the costs to ensure \textit{local} linear convergence, while this paper only requires metric subregularity of the DOT-ADMM operator (which does not imply strong convexity of the problem) and derive a \textit{global} linear convergence result.
In addition, while DOT-ADMM indeed follows the blueprint of \cite{Bastianello21}, it differs in that it allows for local updates to be inexact and the local costs to be time-varying.

\subsection{Proofs of the results}\label{sec:convergence}
The proofs of Theorem~\ref{thm:stochastic+noise+varying} and Theorem~\ref{thm:local} make use of the following general convergence result for stochastic operators enjoying metric-subregularity, which is another original contribution of this manuscript.

\begin{thm}\label{thm:glob-msr}
Let ${\weTs_k:\rea^m\rightarrow \rea^m}$ be a time-varying operator defined component-wise~by
\begin{equation}\label{eq:gen-stoch}
    \weTs_{\ell,k}( z) := \begin{cases}
    \Ts_{\ell,k}( z) + e_{\ell,k} & \text{if } \beta_{\ell,k} = 1 \\
    z_{\ell}& \text{otherwise}
\end{cases}
\end{equation}
for $\ell=1,\ldots,m$, where $e_{\ell,k}$ are i.i.d. random variables and ${\beta_{\ell,k}\sim \emph{\text{Ber}}(p_{\ell})}$ are Bernoulli i.i.d. random variables such that $p_{\ell}\in(0,1]$. If at each time $k\in\nat$ it holds:
\begin{itemize}
    \begingroup
    \medmuskip=0mu
    \thinmuskip=0mu
    \thickmuskip=0mu
    \item[i)] $\exists \varsigma>0$ such that $  \|\proj_{\Ts_k}(z) - \proj_{\Ts_{k-1}}(z)\| \leq \varsigma$, $\forall z \in \rea^m$;
    \endgroup
    \item[ii)] ${\Ts_k}$ is $\alpha$-averaged;
    \item[iii)] $\Ts_k$ is $\gamma$-metric subregular;
\end{itemize}
then the iteration $ z(k)=\weTs_k( z(k-1))$ converges linearly in mean according to 
\begingroup
\medmuskip=1mu
\thinmuskip=1mu
\thickmuskip=1mu
$$
\ev{d_{\Ts_k}( z(k))} \leq \sqrt{\frac{\pmax}{\pmin}} \left[\mu^k d_{\Ts_0}( z(0)) +  \sum_{h=1}^{k}\mu^{k-h}(\ev{\norm{ e_h}}+\mu\varsigma)\right]
$$
\endgroup
where the convergence rate $\mu\in(0,1)$ is given in eq. \eqref{eq:mu} and where $\pmax  = \max_\ell p_\ell$ and $\pmin = \min_\ell p_\ell$ are the maximum and minimum error probabilities. Moreover, it almost surely holds that the iteration asymptotically converges to
$$
\limsup_{k \to \infty} d_{\Ts_k}( z(k)) \leq \lim_{k \to \infty} \sqrt{\frac{\pmax}{\pmin}}\sum_{h = 1}^{k} \mu^{k - h} (\ev{\norm{ e_h}}+\mu\varsigma).
$$
\end{thm}
\begin{proof}
    See Appendix A.
\end{proof}

\begin{rem}
The results of Theorem~\ref{thm:glob-msr} are stronger than most of the literature in that they provide a punctual upper bound on the distance to the set of optimal solutions together with a linear convergence rate, in contrast to other state-of-the-art results, such as those in \cite{combettes_stochastic_2015,bastianello_stochastic_2022}. 
Results in \cite{combettes_stochastic_2015,bastianello_stochastic_2022} only rely on the averagedness property to prove sub-linear convergence of the iteration, where \cite{combettes_stochastic_2015} does not provide a punctual upper bound to the error but only an asymptotic upper bound and where \cite{bastianello_stochastic_2022} provides both using a regret-style metric.
In contrast, our Theorem~\ref{thm:glob-msr} exploits the additional metric-subregularity property to prove linear convergence by using the distance from the set of fixed points as a metric. 
The results of Theorem~\ref{thm:glob-msr} are also much more practical and can be exploited in more realistic scenarios. 
Firstly, we do not assume that the influence of additive errors vanishes over time as in \cite{combettes_stochastic_2015}, instead, we allow for persistent errors due, for instance, to the computational limitations of the agents.
Secondly, we only assume that the fixed point sets at two consecutive iterations are \virg{\textit{similar enough}} and not necessarily overlapping as assumed in \cite{combettes_stochastic_2015}.
\end{rem}

We also provide a result to cover the case in which metric subregularity does not hold in the entire state space, but only in a subspace of it: this property is called locally metric subregularity. 

\begin{thm}\label{thm:loc-msr-static-converr}
In the scenario of Theorem~\ref{thm:glob-msr}, if it holds:
\begin{itemize}
    \item[i)] $\Ts_k=\Ts$ is not time-varying, i.e., $\varsigma = 0$;
    \item[ii)] ${\Ts}$ is $\alpha$-averaged;
    \item[iii)] $\Ts$ is metric subregular in a set $\X \subset \rea^n$;
    \item[iv)] $\lim_{k\rightarrow\infty}\norm{ e_k} \rightarrow 0$;
\end{itemize}
then it almost surely holds
$
\limsup_{k \to \infty} d_{\Ts}( z(k)) = 0.
$
Moreover, there is a finite time $k^\star$ such that for $k\geq k^\star$ $\text{linear}$ convergence in mean is achieved with rate $\mu$ in eq.~\eqref{eq:mu}. 
\end{thm}
\begin{proof}
    See Appendix B.
\end{proof}

Before proceeding with the proofs of our main results, let us conveniently rewrite the operators of the DOT-ADMM updates in eq. \eqref{eq:admm-complete} in compact form as follows\footnote{Note that since $x(k)$ is a function of $z(k-1)$, the update of $z(k)$ depends solely on $z(k-1)$ or, in other words, $x(k)$ is an internal variable of $\text{DOT-ADMM}$.}
\begingroup
\medmuskip=1mu
\thinmuskip=1mu
\thickmuskip=1mu
\begin{equation}\label{eq:compactform}
\small
\begin{aligned}
 x(k) &= \Fs_k( z(k-1)) = \prox_{f_{k}}^{1/\rho \eta}(D A^\top  z(k-1))\\
 z(k) &= \Ts_k( z(k-1)) = \left[ (1 - \alpha) I - \alpha P \right]  z(k-1) + 2 \alpha \rho P A  x(k)
\end{aligned}
\end{equation}
\endgroup
where the operator $\prox_{f_{k}}^{1/\rho \eta} : \rea^{np} \to \rea^{np}$ applies block-wise the proximal of the time-varying local costs $f_{i,k}$; 
the matrix $A\in\rea^{\xi p\times np}$ is given\footnote{The symbol $\otimes$ denotes the Kronecker product.} by $A=\Lambda\otimes I_p$ with ${\Lambda\in\{0,1\}^{\xi\times n}}$ given by ${\Lambda = \operatorname{blk\,diag}\{ \bone_{\eta_i}\}_{i = 1}^n}$;
the matrix $D\in\rea^{np\times np}$ is given by $D = \operatorname{blk\,diag}\{ (\rho\eta_i)^{-1} I_p \}_{i = 1}^n$;
the matrix ${P\in \{0,1\}^{\xi p \times \xi p} }$ is given by $P=\Pi\otimes I_p$ with $\Pi\in\{0,1\}^{\xi\times \xi}$ being a permutation matrix swapping $(i,j)\in\E$ with $(j,i)\in\E$.

\subsubsection{Proof of Theorem \ref{thm:stochastic+noise+varying}}
By~\cite[Proposition~3]{Bastianello21}, for each fixed point $ z^\star_k \in \fix(\Ts_k)$ there is $ x_k^\star=\Fs_k( z_k^\star)$ which is a solution to the problem in eq. \eqref{eq:online-distributed-optimization}. Thus, letting $\X_k^\star$ be the time-vaying set of solutions, we can write
\begingroup
\medmuskip=0mu
\thinmuskip=0mu
\thickmuskip=0mu
\begin{align*}
    d_{\X^\star_k}( x(k)) &=  \inf_{ y \in \X^\star_k} \norm{ x(k) -  y}  \overset{(i)}{\leq} \norm{ x(k) -  x^\star_k} \\
    &= \norm{\Fs_k( z(k-1)) - \Fs_k( z_k^\star)} \\
    &= \norm{\prox_{f_{k}}^{1/\rho \eta}(D A^\top  z(k-1)) - \prox_{f_{k}}^{1/\rho \eta}(D A^\top  z^\star_k)} \\
    &\overset{(ii)}{\leq} \norm{D A^\top ( z(k-1) -  z^\star_k)} \leq \norm{D A^\top} \norm{ z(k-1) -  z^\star_k} \\
    &\overset{(iii)}{=} \norm{DA^\top} d_{\Ts_k}( z(k-1))
\end{align*}
\endgroup
where $(i)$ holds since $ x_k^\star\in\X_k^\star$, $(ii)$ follows by the non-expansiveness of the proximal, and $(iii)$ holds by choosing
$$ z^\star_k =\arginf_{ y \in \fix(\Ts_k)} \mnorm{ z(k-1) -  y}.
$$
This means that the linear convergence of $ x(k)$ to a neighborhood of $\X^\star_k$ is implied by that of $ z(k)$ to a neighborhood of $\fix(\Ts_k)$, which can be proved by means of Theorem~\ref{thm:glob-msr}.
Indeed, by Assumptions~\ref{as:random-updates}-\ref{as:additive-error} the update of $ z(k)$ can be described by a stochastically
perturbed operator as in eq. \eqref{eq:gen-stoch} object of Theorem~\ref{thm:glob-msr}.
We thus prove both statements of the theorem by checking all the conditions under which Theorem~\ref{thm:glob-msr} holds:
\begin{itemize}
    \item[\textit{i)}] By definition, the fixed points of DOT-ADMM are
    $$
        \fix(\Ts_k) = \{ z \ | \ (I + P) z = 2 \rho P A x^\star_k, \ x^\star_k \in \X_k^\star \}.
    $$
    Therefore the projection of a point $z$ onto $\fix(\Ts_k)$ is \cite[section~6.2.2]{Parikh14}
    $$
        \proj_{\Ts_k}(z) = z - (I + P)^\dagger \left( (I + P) z - 2 \rho P A x^\star_k \right).
    $$
    With some simple algebra, we can then see that
    \begingroup
    \medmuskip=0mu
    \thinmuskip=0mu
    \thickmuskip=0mu
    \small
    \begin{align*}
        \norm{\proj_{\Ts_k}(z) - \proj_{\Ts_{k-1}}(z)} &= 2 \rho \norm{(I + P)^\dagger P A (x^\star_k - x^\star_{k-1})} \\
        & \overset{(i)}{\leq} 2 \rho \norm{(I + P)^\dagger P A} \sigma
    \end{align*}
    \endgroup
    where (i) follows by sub-multiplicativity of the norm and Assumption~\ref{as:time-variability}. Thus assumption~\textit{i)} of Theorem~\ref{thm:glob-msr} is verified for $\varsigma = 2 \rho \norm{(I + P)^\dagger P A} \sigma$.

    \item[\textit{ii)}] By Assumption~\ref{as:convexity} it follows that $\Ts_k$ are $\alpha$-averaged for all $k \in \nat$. Indeed, the operator $\Ts_k$ comes from the application of the Peaceman-Rachford operator to the dual problem of \eqref{eq:online-distributed-optimization}, which guarantees its $\alpha$-averagedness when the cost functions $f_{i,k}$ are convex (cfr.~\cite{Bastianello21});
        
    \item[\textit{iii)}] $\Ts_k$ is $\gamma$-metric subregular by assumption.
    \end{itemize}

\subsubsection{Proof of Theorem~\ref{thm:local}}
Since the local costs are strongly convex and twice differentiable, one can approximate the local costs by quadratic functions of the following kind
\begin{align*}
    f_i(x) &= \frac{1}{2} (x - x^\star)^\top \nabla^2 f_i(x^\star) (x - x^\star) \\ &+ \langle \nabla f_i(x^\star), x - x^\star \rangle + o(x - x^\star)
\end{align*}
where
$$
x^\star = \argmin_x \sum_{i = 1}^nf_i(x)\quad \text{and}\quad \lim_{x\rightarrow x^\star} \frac{\norm{o(x - x^\star)}}{\norm{x - x^\star}} = 0
.$$

Therefore, one can interpret DOT-ADMM as being characterized by an affine operator with an additive error that depends on the higher order terms $o(x - x^\star)$. But since affine operators are metric subregular~\cite{robinson1981some,Themelis19}, then the DOT-ADMM operator is metric subregular around the optimal solution $x^\star$. Finally, since the additive error vanishes around $x^\star$, then we can apply Theorem~\ref{thm:loc-msr-static-converr} (and, specifically, the particular cases outlined in Corollary \ref{cor:stochastic}) and prove that linear convergence can be achieved locally in mean square.

\section{Tutorial examples:\\ linear and logistic regression}\label{sec:distributed-application}
This section discusses some tutorial examples of distributed learning problems for which DOT-ADMM is characterized by a metric subregular operator, and, as a consequence, Theorems~\ref{thm:stochastic+noise+varying} and \ref{thm:local} apply.
For simplicity, the discussion is limited to the case of static local costs even though all the results apply straightforwardly to the online scenario with time-varying local costs.
We consider \textit{empirical risk minimization} (ERM) problems, in which the local cost of each agent $i \in \V$ is defined over the data set $\{a_{i,h},b_{i,h}\}_{h=1}^{m_i}$, $m_i \in \nat$:
\begin{equation}\label{eq:erm-costs}
    f_{i}(x) = \sum_{h = 1}^{m_i} g(x,a_{i,h},b_{i,h})
\end{equation}
where $g \in \Gamma_0^p$ is a suitable \textit{loss function}.
The goal of an ERM problem is that of computing online a solution $x^\star_k\in \rea^p$ to eq. \eqref{eq:online-distributed-optimization} with local costs in eq. \eqref{eq:erm-costs} where $x^\star_k$ represents the vector of trained parameters of a model.
Such a goal can be reached by employing DOT-ADMM, whose updates are ruled by the operator $\Ts$ in \eqref{eq:compactform}, recalled next
\begingroup
\medmuskip=0mu
\thinmuskip=0mu
\thickmuskip=0mu
\begin{equation}\label{eq:compact-z}
    \Ts( z) = \left[ (1 - \alpha) I - \alpha P \right]  z + 2 \alpha \rho P A \prox_{f}^{1/\rho \eta}(D^{-1} A^\top  z)
\end{equation}
\endgroup
where $\prox_{f}^{1/\rho \eta} : \rea^{np} \to \rea^{np}$ applies block-wise the proximal of the local costs $f_{i}$.
The specific structure of the operator $\Ts$ allows for \textit{regularized} versions of the local costs $f_i(x) + \frac{\epsilon}{2} \norm{x}^2$ which only results in a scaling of the proximal, i.e. (cfr. \cite[section~2.2]{Parikh14}),
$$
    \prox_{f_i + \frac{\epsilon}{2} \norm{\cdot}^2}^\rho(x) = \prox_{f_i}^{1 / (\epsilon + 1 / \rho)}\left( x / (1 + \rho \epsilon) \right).
$$

To apply Theorems~\ref{thm:stochastic+noise+varying} and \ref{thm:local} one needs to prove that $\Ts$ is metric subregular given a specific loss function $g\in \Gamma_0^p$.
To this end, we will make use of the following novel result, which provides an \textit{operative way to verify metric subregularity of an operator}.
We anticipate here that, differently from linear regression problems, the proximal of robust linear regression and logistic regression costs do not admit a closed-form solution.
Nevertheless, resorting to the following Proposition~\ref{prop:ULaffMS} it is possible to show metric subregularity indirectly.

{
\begin{prop}\label{prop:ULaffMS}
An operator $\Ts:\mathbb{R}^n\rightarrow \mathbb{R}^n$ is metric subregular if there is a matrix $A\in\rea^{n\times n}$ and two offsets $b_\Ls,b_\Us\in\rea^n$ such that $\Ts$ is lower and upper bounded (component-wise) by the affine operators $\Ls(z):=Az+b_\Ls$ and $\Us(z):=Az+b_\Us$, respectively, yielding
\begin{equation}\label{eq:ULaffMS}
\Ls(z) \leq \Ts(z)\leq \Us(z),\qquad \forall z\in\mathbb{R}^n.
\end{equation}
\end{prop}
}

\begin{proof}
See Appendix C.
\end{proof}

\vspace{-0.5em}

\subsection{Linear regression}

In linear regression problems, the data sets are such that $a_{i,h}\in\rea^p$ and $b_{i,h}\in\rea$ and the loss function $g$ is given by 
\begin{equation}\label{eq:linear-regression}
    g(x,a_{i,h},b_{i,h}) = \frac{1}{2} (a_{i,h}^\top x - b_{i,h})^2.
\end{equation}
The following result holds for this class of problems.
\begin{prop}\label{prop:linear}
Consider the operator $\Ts$ in~\eqref{eq:compact-z} characterizing DOT-ADMM applied to a linear regression problem \ref{eq:online-distributed-optimization}, that is, with local costs~\eqref{eq:erm-costs} and loss~\eqref{eq:linear-regression}. Then, $\Ts$ is metric subregular.
\end{prop}

\begin{proof}
Denoting ${A_{i} = [a_{i,1},\cdots,a_{i,m_i}]^\top \in\rea^{m_i\times p} }$ and ${b_{i}=[b_{i,1},\cdots,b_{i,m_i}]^\top\in\rea^{m_i}}$ the local costs become $f_{i}(x) = \frac{1}{2} \norm{A_{i}x_i-b_{i}}^2$.
In this particular case, the proximals of the local costs have the following closed-form expression: 
\begingroup
\medmuskip=2mu
\thinmuskip=2mu
\thickmuskip=2mu
\begin{equation}\label{eq:proximal-quadratic}
    \prox_{f_{i}}^{1/\rho\eta_i}(w) =  (A_{i}^\top A_{i}+\rho\eta_i I)^{-1} \left(\rho\eta_i w + A_{i}^\top b_{i}\right).
\end{equation}
\endgroup

By noticing that the proximals are affine functions of their argument $w$, it follows that also the operator $\Ts$ in~\eqref{eq:compact-z} is affine.
Consequently, $\Ts$ is metric subregular by Proposition \ref{prop:ULaffMS}.
\end{proof}

\subsection{Robust linear regression}
Linear regression may be sensitive to outliers when using a quadratic loss. To remedy this, it is possible to formulate a \textit{robust linear regression} problem by using the \textit{Huber loss} in the local costs~\eqref{eq:erm-costs}:
\begingroup
\medmuskip=1mu
\thinmuskip=1mu
\thickmuskip=1mu
\begin{equation}\label{eq:huber-loss}\small
    g(x,a_{i,h},b_{i,h}) = \begin{cases}
        \frac{1}{2} (a_{i,h}^\top x - b_{i,h})^2 & \text{if} \ |a_{i,h}^\top x - b_{i,h}| \leq \theta \\
        \theta ( |a_{i,h}^\top x - b_{i,h}| - \frac{\theta}{2} ) & \text{otherwise}
    \end{cases}
\end{equation}
\endgroup
with $\theta > 0$.
The following result holds for this class of problems.
\begin{prop}\label{prop:robust-linear}
Let $\Ts$ be the operator characterizing DOT-ADMM applied to a robust linear regression problem, that is, with local costs~\eqref{eq:erm-costs} and loss~\eqref{eq:huber-loss}. Then, $\Ts$ is metric subregular.
\end{prop}

\begin{proof}
The proximal of the local cost $f_i$ is
\begingroup
\medmuskip=0mu
\thinmuskip=0mu
\thickmuskip=0mu
$$\small
\prox^{1/\rho\eta_i}_{f_i}(w)=\argmin_{ x} \underbrace{\left\{\sum_{h = 1}^{m_i} g(x,a_{i,h},b_{i,h})+\frac{1}{2\rho\eta_i}\norm{x-
 w}^2\right\}}_{h(x)}.
$$
\endgroup
Thus, the proximal is the (unique) stationary point of $h(x)$, which is the solution of
$$
\frac{\partial}{\partial x}h(x) =
\sum_{h = 1}^{m_i} \frac{\partial}{\partial x} g(x,a_{i,h},b_{i,h}) + \frac{1}{\rho\eta_i}(x-w)= 0,
$$

Since by the definition of the Huber loss~\eqref{eq:huber-loss} it holds
$$
\abs{\frac{\partial}{\partial x} g(x,a_{i,h},b_{i,h})} \leq \theta \abs{a_{i,h}},\qquad \forall x\in\rea^p,
$$
it follows that
$$
w - \rho \eta_i \theta \abs{a_{i,h}} \leq \prox^{1/\rho\eta_i}_{f_i}(w)\leq w +\rho \eta_i \theta \abs{a_{i,h}}.
$$
i.e., the proximal is upper/lower bounded by the identity operator with offsets $\pm \rho \eta_i \theta \abs{a_{i,h}}$. 
Consequently, also the operator $\Ts$ is bounded by two offsetted operators and we can apply Proposition \ref{prop:ULaffMS} to prove that is metric subregular.
\end{proof}

\subsection{Logistic regression}
We turn now to classification problems using logistic regression. The data sets in this case are such that $a_{i,h}\in\rea^p$ and $b_{i,h}\in\{-1,1\}$, with the loss
\begin{equation}\label{eq:logistic-regression}
    g(x,a_{i,h},b_{i,h}) = \log\left( 1 + \exp\left( - b_{i,h} a_{i,h}^\top  x \right) \right).
\end{equation}
The following result holds for this class of problems.
\begin{prop}\label{prop:logistic}
Let $\Ts$ be the operator characterizing DOT-ADMM applied to a logistic regression problem, that is, with local costs~\eqref{eq:erm-costs} and loss~\eqref{eq:logistic-regression}. Then, $\Ts$ is metric subregular.
\end{prop}
\begin{proof}
By definition of the logistic loss~\eqref{eq:logistic-regression}, it holds
$$
    \abs{\frac{\partial}{\partial x} g(x,a_{i,h},b_{i,h})} = \abs{b_{i,h} a_{i,h}}.
$$
The proof follows by similar arguments as those in Proposition~\ref{prop:robust-linear}.
\end{proof}

\section{Numerical results}\label{sec:numerical}
In this section, we carry out numerical simulations corroborating the theoretical results of the previous sections; all simulations have been implemented in Python using the \texttt{tvopt} package~\cite{bastianello_tvopt_2021}, and run on a laptop with 12$^\text{th}$ generation Intel i7 CPU and $16$ GB of RAM.

The considered set-up is a network of $N=10$ nodes, exchanging information through a random graph topology of $20$ edges, that want to solve an online logistic regression problem characterized by~\eqref{eq:online-distributed-optimization} and the local costs
$$
    f_{i,k}(x) = \sum_{h = 1}^{m_i} \log\left( 1 + \exp\left( - b_{i,h,k} a_{i,h,k} x \right) \right) + \frac{\epsilon}{2} \norm{x}^2
$$
where $x \in \rea^{p}$, $p = 16$, is the vector of weights and intercept to be learned, and $\{ (a_{i,h,k}, b_{i,h,k} \}_{h = 1}^{m_i}\}$ with $a_{i,h,k} \in \rea^{1 \times p}$, $b_{i,h,k} \in \{ -1, 1 \}$ are the $m_i = 20$ feature vectors and class pairs available to the node at time $k \in \nat$. Notice that we add a regularization term ($\epsilon = 5$) to ensure strong convexity.

In the following sections, we discuss the performance of DOT-ADMM in the different scenarios presented in Section~\ref{subsec:challenges}. The algorithm will then be compared to the gradient tracking methods~\cite{bof_multiagent_2019} (designed to be robust to asynchrony), and~\cite{liu_linear_2021} (designed to be robust to quantization).

\begin{table}[!b]
\begin{center}
\caption{Computational time of local updates and asymptotic error for a static logistic regression problem.}
\label{tab:prox-computation-admm}
\begin{tabular}{ccc}
    \\ \hline
    Threshold               & Comp. time [s]            & Asymptotic err.           \\
    \hline
    $\theta = 10^{-14}$     & $3.47 \times 10^{-3}$     & $4.14 \times 10^{-14}$    \\
    $\theta = 10^{-12}$     & $2.84 \times 10^{-3}$     & $3.65 \times 10^{-12}$    \\
    $\theta = 10^{-10}$     & $2.42 \times 10^{-3}$     & $4.88 \times 10^{-10}$    \\
    $\theta = 10^{-8}$      & $1.95 \times 10^{-3}$     & $5.30 \times 10^{-8}$     \\
    $\theta = 10^{-6}$      & $1.39 \times 10^{-3}$     & $1.01 \times 10^{-5}$     \\
    $\theta = 10^{-4}$      & $8.88 \times 10^{-4}$     & $5.73 \times 10^{-4}$     \\
    $\theta = 10^{-2}$      & $4.12 \times 10^{-4}$     & $9.71 \times 10^{-2}$     \\
    \hline
\end{tabular}
\end{center}
\end{table}

\subsection{Local updates for logistic regression}\label{subsec:local-updates}
While running DOT-ADMM, each active node needs to compute the local update~\eqref{eq:admm-x}. However, when applied to logistic regression in eq.~\eqref{eq:logistic-regression} the proximal of $f_{i,k}$ does not have a closed form -- differently from the linear regression problem~\eqref{eq:linear-regression} -- and therefore, the proximal needs to be computed approximately.
In our set-up, a node computes an approximation of~\eqref{eq:admm-x} via the accelerated gradient descent, terminating when the distance between consecutive iterates is smaller than a threshold $\theta > 0$. The error introduced by such an inexact local update is smaller the smaller $\theta$ is. However, smaller values of the threshold make the computational time required for a local update longer, presenting a trade-off.

To exemplify this trade-off, we apply DOT-ADMM to a static version of~\eqref{eq:logistic-regression}. In Table~\ref{tab:prox-computation-admm} we report the computational time required to compute the local updates for different choices of $\theta$, as well as the corresponding asymptotic error (that is, the distance $\norm{x(k) - x^\star}$ from the unique solution at the end of the simulation). The computational time is computed by averaging over $250$ iterations of the algorithm.

Hereafter, unless otherwise stated we use DOT-ADMM with $\theta = 10^{-8}$, for a local update time of ${\sim 2.42 \times 10^{-3} s}$. For comparison, we note that a local update of DGT (with hand-tuned parameters to improve performance) requires $\sim 1.90 \times 10^{-3} s$, and in the simulations we allow DGT to run two iterations per each iteration of DOT-ADMM, to account for the longer time required in the latter local updates.

\begin{table}[!b]
\begin{center}
\caption{Asymptotic error for different quantization levels.}
\label{tab:quantization}
\begin{tabular}{cc}
    \\ \hline
    Quantization            & Asymptotic error         \\
    \hline
    No quantization         & $5.30 \times 10^{-8}$    \\
    $\delta = 10^{-10}$     & $5.30 \times 10^{-8}$    \\
    $\delta = 10^{-8}$      & $7.36 \times 10^{-8}$    \\
    $\delta = 10^{-6}$      & $4.74 \times 10^{-6}$    \\
    $\delta = 10^{-4}$      & $5.64 \times 10^{-4}$    \\
    $\delta = 10^{-2}$      & $5.32 \times 10^{-2}$    \\
    $\delta = 10^{-1}$      & $4.91 \times 10^{-1}$    \\
    \hline
\end{tabular}
\end{center}
\end{table}

\subsection{Quantized communications}
As in the section above, we consider a static logistic regression problem, and assume that the agents can exchange quantized communications. In particular, an agent $i$ can only send the quantized version $q(x)$ of a message $x \in \rea^p$, as defined component-wise by
$$
    [q(x)]_j = \begin{cases}
        \ubar{q}                                 & \text{if} \ [x]_j < \ubar{q} \\
        \delta \lfloor [x]_j / \delta \rfloor    & \text{if} \ \ubar{q} \leq [x]_j \leq \bar{q} \\
        \bar{q}                                  & \text{if} \ [x]_j > \bar{q} \\
    \end{cases}, \quad j \in \{ 1, \ldots, p \}
$$
with $\bar{q} = -\ubar{q} = 10$, and $\delta > 0$ the quantization level.
Table~\ref{tab:quantization} reports the asymptotic error of DOT-ADMM for different quantization levels $\delta$.

\subsection{Asynchrony}\label{subsec:asynchrony}
In Section~\ref{subsec:local-updates} we discussed how the local updates~\eqref{eq:admm-x} for a logistic regression problem need to be computed recursively as the proximal does not have a closed form solution. We then discussed how the threshold specifying the accuracy of the local update impacts the convergence of the algorithm.
Here we consider how recursive local updates can lead to \textit{asynchronous operations} of the agents, due to their heterogeneous computational capabilities.

We consider the following scenario: at iteration $k$ each agent completes the local update~\eqref{eq:admm-x} -- using $\theta = 10^{-8}$ -- with some probability, $\underline{p}$ or $\bar{p}$, $\underline{p} < \bar{p}$. The agents characterized by the smaller probability $\underline{p}$ are the ``slow'' nodes, which, having fewer computational resources, take on average a longer time to reach the threshold $\theta$. Notice that all the nodes use the same threshold, and their more or less frequent updates mimic the effect of different resources.

In Figure~\ref{fig:asynchrony} we report the mean tracking error (as averaged over $100$ Monte Carlo iterations) for the asynchronous case with different numbers of slow nodes $N_s$. We also compare the result with the error in the synchronous case, in which all nodes complete an update at each iteration $k$.
As discussed in Section \ref{sec:convrate}, asynchronous agent operations, which translate into random coordinate updates, lead to worse convergence rates. Indeed, the more frequent the updates are, the faster the convergence rate (until achieving that of the synchronous version), and the introduction of slower nodes implies less frequent coordinate updates overall.

\begin{figure}[!t]
\centering
\includegraphics[width=0.45\textwidth]{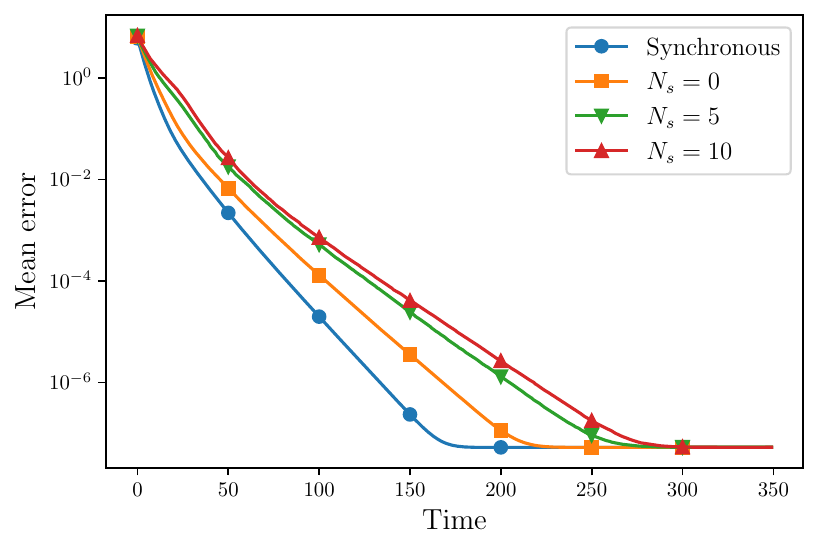}
\caption{Error trajectories of DOT-ADMM with synchronous and asynchronous updates for different numbers of slow nodes.}
\label{fig:asynchrony}
\end{figure}

\begin{figure}[!t]
\centering
\includegraphics[width=0.48\textwidth]{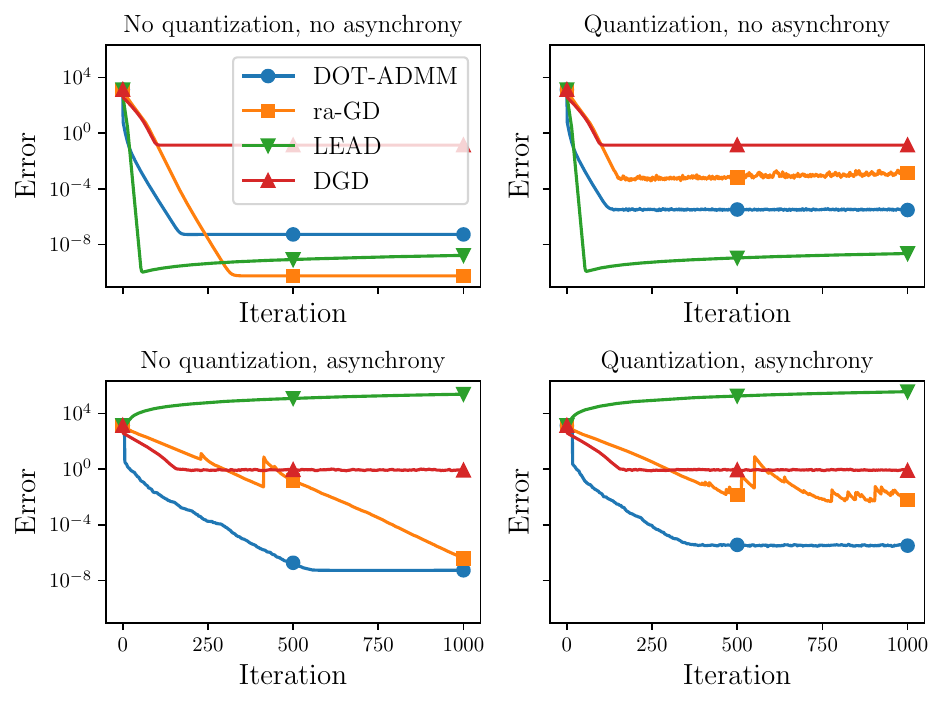}
\caption{Comparison on static problems of DOT-ADMM with ra-GD, LEAD, DGD in different scenarios combining quantization/asynchrony.}
\label{fig:comparison}
\end{figure}

\subsection{Online optimization}\label{sec:numerical-online}
In this section, we evaluate the performance of DOT-ADMM when applied to two instances of the online logistic regression problem, in which the local cost functions are piece-wise constant. Specifically, the costs change $10$ and $100$ times, respectively, and are generated so that the maximum distance between consecutive optima is $\sim 2.5$ (cf. Assumption~\ref{as:time-variability}). In Figure~\ref{fig:online} we report the tracking error of DOT-ADMM when applied to the two problems.
Notice that when the problem changes less frequently, DOT-ADMM has time to converge to smaller errors, up to the bound imposed by the inexact local updates (computed with $\theta = 10^{-4}$). Notice that in the transient the convergence is linear, as predicted by the theory. On the other hand, more frequent changes in the problem yield larger tracking errors overall.

\begin{figure}[!t]
\centering
\includegraphics[width=0.45\textwidth]{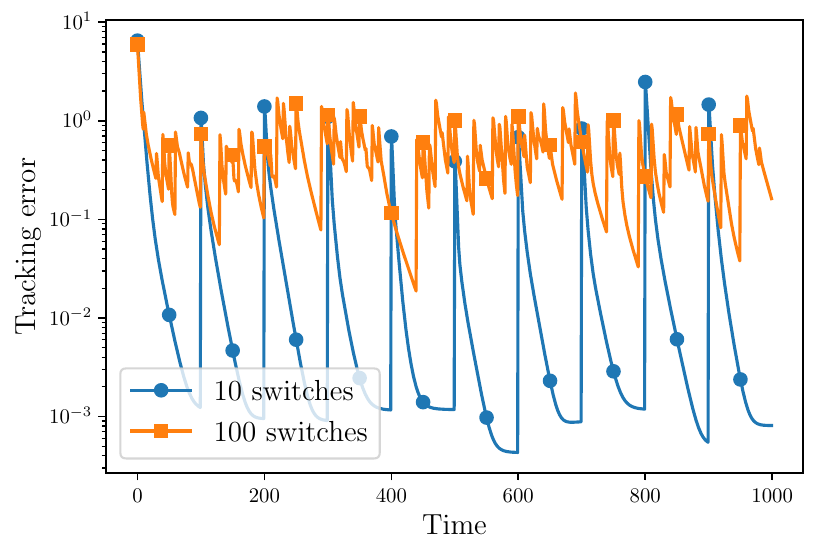}
\caption{Tracking error of DOT-ADMM applied to two online problems with different piece-wise constant cost functions.}
\label{fig:online}
\end{figure}
\begin{figure}[!t]
\centering
\includegraphics[width=0.48\textwidth]{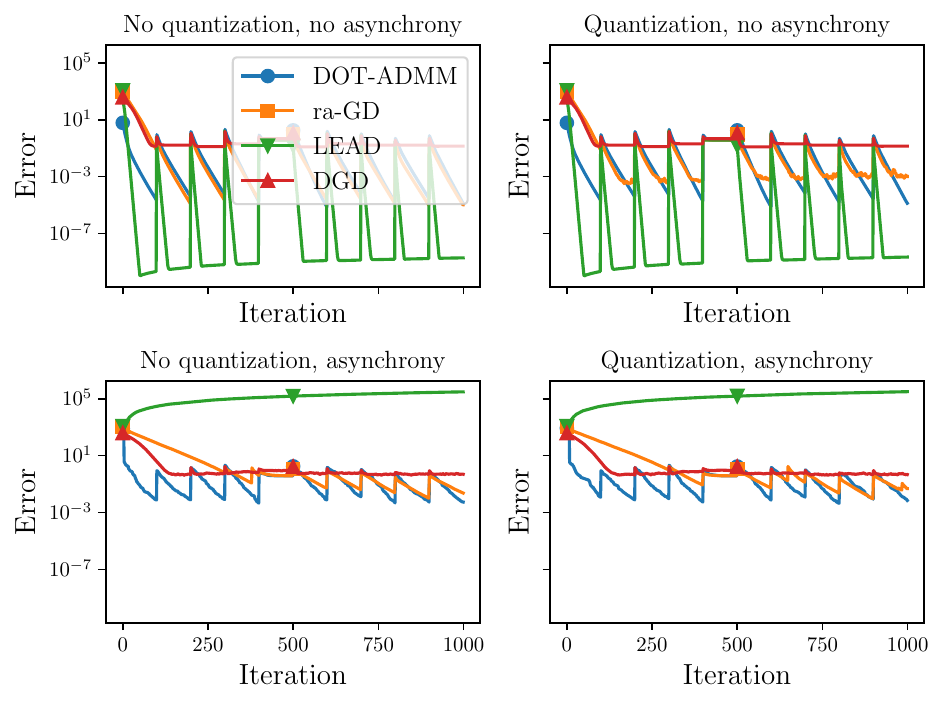}
\caption{Comparison on online problems of DOT-ADMM with ra-GD, LEAD, DGD in different scenarios combining quantization/asynchrony.}
\label{fig:comparison-online}
\end{figure}

\subsection{Comparison with state-of-the-art algorithms}
We conclude by comparing DOT-ADMM with three gradient-based methods, for both static and online problems:
\begin{itemize}
    \item \textit{ra-GD}~\cite{bof_multiagent_2019}\footnote{The paper~\cite{bof_multiagent_2019} proposes a distributed Newton method, but ra-GD can be derived by replacing the Hessians with identity matrices (cf.~\cite[Remark~IV.1]{bof_multiagent_2019}).}: {gradient tracking algorithm} which makes use of the robust ratio consensus to ensure convergence in the presence of asynchrony,

    \item \textit{LEAD}~\cite{liu_linear_2021}: {gradient tracking algorithm} which is designed to be robust to a certain class of unbiased quantizers,

    \item \textit{DGD}~\cite{yuan_convergence_2016}: which does not converge exactly, but has been shown to be robust to additive errors \cite{bastianello_distributed_2021} and online scenarios \cite{simonetto_distributed_2014}, see also \cite{yuan_can_2020}.
\end{itemize}
Due to the fact that DOT-ADMM requires a longer time to update the local states (cf. Section~\ref{subsec:local-updates}), ra-GD, LEAD, and DGD were run for a larger number of iterations to match the computational time of DOT-ADMM. All the step-sizes of these gradient methods were hand-tuned for optimal performance.

In Figure~\ref{fig:comparison} we compare the four algorithms on a static logistic regression problem, and for different scenarios combining quantization and asynchrony. In particular, we either use or not the quantizer~\cite[eq.~(14)]{liu_linear_2021}, and the agents either activate synchronously or asynchronusly, using the same set-up as Section~\ref{subsec:asynchrony}.
In accordance with the theory, ra-GD is robust to asynchrony, although the convergence is somewhat slow due to a necessarily conservative step-size choice. On the other hand, when quantization is employed the algorithm seems to converge only to a neighborhood of the optimal solution, which is larger than the neighborhood reached by DOT-ADMM (despite the fact that DOT-ADMM also uses inexact updates).
As predicted, LEAD shows convergence in the presence of quantization; however, the algorithm is not robust to asynchrony and seems to diverge when the agents are not synchronized.
Finally, DGD is robust to both quantization and asynchrony, but its inherent inexactness leads to poor performance.

In Figure~\ref{fig:comparison-online} we further compare these algorithms for the online logistic regression problem of Section~\ref{sec:numerical-online}. As we can see the performance of the different algorithms while the costs are not varying largely follows that depicted in Figure~\ref{fig:comparison}, with an increase in the error any time the problem changes.

From both Figure~\ref{fig:comparison} and Figure~\ref{fig:comparison-online} we see that only DOT-ADMM shows robustness to all the three challenges of asynchrony, quantization, and time-varying costs.

\section{Conclusions}\label{sec:conclusion}

This paper proposes DOT-ADMM to solve online learning problems in a multi-agent setting under challenging network constraints, such as asynchronous and inexact agent computations, and unreliable communications.
The convergence and robustness of DOT-ADMM have been proven by deriving novel theoretical results in stochastic operator theory for the class of metric-subregular operators, which turns out to be an important class of operators that shows linear convergence to the set of optimal solutions.
The broad applicability of this class of operators is supported by the fact that the operator ruling DOT-ADMM applied to the standard linear and logistic regression problems is indeed metric subregular.
Future works will focus on studying the optimal design of DOT-ADMM and on the characterization of the linear rate of convergence for specific distributed problems, e.g. online learning and dynamic tracking.

\bibliographystyle{IEEEtran}
\bibliography{autosam}

\section*{Appendix B: Proof of Theorem \ref{thm:glob-msr}}\label{app:glob-msr}

\textit{[Punctual upper bound] - } We make use of the so-called \textit{diagonally-weighted norm} in the sense of~\cite{FB-CTDS}, where the vector of positive weights is the vector of probabilities $p=[p_1,\ldots,p_m]^\top$, which is defined next
\begin{equation}\label{eq:newnorm}
    \mnorm{ z}^2 = \sum_{\ell = 1}^m \frac{1}{p_\ell} z_\ell^2,
\end{equation}
following the usual notation in our community~\cite{combettes_stochastic_2015}. 
Clearly, such a norm satisfies the following, where recall that $\pmax = \max_\ell p_\ell$, $\pmin = \min_\ell p_\ell$
\begin{equation}\label{eq:equivalent-norms}
    \pmin \mnorm{ z}^2 \leq \norm{ z}^2 \leq \pmax \mnorm{ z}^2.
\end{equation}
Similarly to the Euclidean distance $d_{\Ts_k}( z)$ from the set of fixed points of $\Ts$, we define the distance
${d_{\Ts_k}'( z) = \inf_{ y \in \fix(\Ts_k)} \mnorm{ z -  y}}$, such that
\begin{equation}\label{eq:metric-sub-mnorm}
    \frac{1}{\pmax} d^2_{\Ts_k}( z) \overset{(i)}{\leq} d'^2_{\Ts_k}( z) \overset{(i)}{\leq} \frac{1}{\pmin} d^2_{\Ts_k}( z) \overset{(ii)}{\leq} \frac{\gamma^2}{\pmin}  \norm{(\Is - \Ts_k) z}^2,
\end{equation}
where (i) follow by the definition of projection and eq.~\eqref{eq:equivalent-norms}, whereas (ii) by $\gamma$-metric subregularity of $\Ts_k$.

We also conveniently rewrite the operator $\weTs_k$ in eq. \eqref{eq:gen-stoch} by 
$$
 z_{\ell}(k)=\weTs_{\ell,k} ( z(k-1)) =\wTs_{\ell,k}( z(k-1)) + \beta_{\ell,k} e_{\ell,k},
$$
where
$$
\wTs_{\ell,k}( z(k-1)) = z_{\ell}(k-1) + \beta_{\ell,k} (\Ts_{\ell,k} ( z(k-1)) - z_{\ell}(k-1)).
$$
Letting $e_k\in\rea^m$ be the vector stacking all the errors and ${ z^\star_k \in \fix(\Ts_k)}$, then by eq. \eqref{eq:metric-sub-mnorm} and the triangle inequality we can write
\begingroup
\medmuskip=0mu
\thinmuskip=0mu
\thickmuskip=0mu
\begin{align*}
	d'_{\Ts_k}( z(k)) & =  
 \mnorm{\weTs_k( z(k-1)) -  z^\star_k} \leq  \mnorm{\wTs_k( z(k-1)) -  z^\star_k} + \mnorm{ e_k}
\end{align*}
\endgroup
and thus
$$
\ev{d'_{\Ts}( z(k))} \leq  \ev{\mnorm{\wTs_k( z(k-1)) -  z^\star_k}} +\ev{\mnorm{ e_k}}.
$$
We are interested in finding an upper-bound to the first term on the right-hand side of the above inequality, whose explicit form is given by $\mnorm{\wTs_k( z(k-1)) - z^\star_k}^2 =$
\begingroup
\medmuskip=0mu
\thinmuskip=0mu
\thickmuskip=0mu
\begin{align*}
	&= \sum_{\ell=1}^m \frac{1}{p_{\ell}} \Big[(1-\beta_{{\ell},k}) z_{\ell}(k-1) + \beta_{\ell,k} \Ts_{\ell,k}( z(k-1)) - z_{\ell,k}^\star) \Big]^2 \\
	&= \sum_{\ell=1}^m\Big[\frac{1-\beta_{\ell,k}}{p_{\ell}} (z_{\ell}(k-1) - z_{\ell}^\star)^2 + \frac{\beta_{\ell}}{p_{ij}} (\Ts_{\ell,k}( z(k-1)) - z_{\ell,k}^\star)^2 \Big]
\end{align*}
\endgroup
where, since $\beta_{\ell,k} \in \{ 0, 1 \}$, we have used the following: ${\beta_{\ell,k}^2 = \beta_{\ell,k}}$; ${(1-\beta_{\ell,k})^2 = (1-\beta_{\ell,k})}$; ${(1-\beta_{\ell,k}) \beta_{\ell,k} = 0}$.
An upper-bound to the conditional expectation w.r.t. the realizations of all r.v.s at time $k-1$ is given by
\begingroup
\medmuskip=1mu
\thinmuskip=1mu
\thickmuskip=1mu
\allowdisplaybreaks
\begin{align*}
	&\evc{k-1}{\mnorm{\wTs_k( z(k-1)) -  z^\star_k}^2} = \\
	&= \sum_{\ell=1}^m \left[\frac{1-p_\ell}{p_\ell} (z_\ell(k-1) - z_{\ell,k}^\star)^2 + (\Ts_{\ell,k}( z(k-1)) - z_{\ell,k}^\star)^2 \right] \\
	&= \mnorm{ z(k-1) -  z^\star_k}^2 - \norm{ z(k-1) -  z^\star_k}^2  + \norm{\Ts_k( z(k-1)) - z^\star_k}^2 \\
	&\overset{(i)}{\leq} \mnorm{ z(k-1) -  z^\star_k}^2 - \frac{1-\alpha}{\alpha} \norm{(\Is - \Ts_k)  z(k-1)}^2 \\
    &\overset{(ii)}{\leq}  d'^2_{\Ts_k}( z(k-1)) -\frac{1-\alpha}{\alpha} \norm{(\Is - \Ts_k)  z(k-1)}^2 \\
	&\overset{(iii)}{\leq} d'^2_{\Ts_k}( z(k-1)) - \frac{1-\alpha}{\alpha\gamma^2}  \pmin d'^2_{\Ts_k}( z(k-1))  \\
	&= \left(1 - \frac{(1-\alpha)\pmin}{\alpha\gamma^2} \right)d'^2_{\Ts_k}( z(k-1)) := \mu^2 d'^2_{\Ts_k}( z(k-1))
\end{align*}
\endgroup
where (i) holds by $\alpha$-averagedness, (ii) follows by selecting $ z^\star_k=\arginf_{ y \in \fix(\Ts_k)} \mnorm{ z(k-1) -  y}$, and (iii) is a consequence of metric subregularity highlighted in eq.~\eqref{eq:metric-sub-mnorm}, and where $\mu \in (0, 1)$ provided that $\gamma$ is sufficiently large, which we can always assume by overestimating the metric subregularity constant of $\Ts_k$, in particular by replacing $\gamma^2$ with $\lambda=\max\{\gamma^2, (1-\alpha)\pmin/\alpha\}$ as in eq. \eqref{eq:mu}.
Now, exploiting the (i) concavity of the square root, the (ii) Jensen's inequality and the (iii) law of total expectation, we have
\begingroup
\medmuskip=1mu
\thinmuskip=1mu
\thickmuskip=1mu
\begin{equation*}\label{eq:square-expected-value}
	\mathbb{E}\Big[\mnorm{\cdot}\Big] \overset{(i)}{=} \ev{\sqrt{\mnorm{\cdot}^2}} \overset{(ii)}{\leq} \sqrt{\ev{\mnorm{\cdot}^2}} \overset{(iii)}{=} \sqrt{\ev{\evc{k-1}{\mnorm{\cdot}^2}}}
\end{equation*}
\endgroup
which implies
$$
\ev{\mnorm{\wTs_k( z(k-1))- z^\star}}\leq \mu \ev{d_{\Ts}( z(k-1))}.
$$

Let us now combine the definition of $d'_{\Ts_k}$, assumption \textit{i)}, and the triangle inequality to derive the following bound
\begin{align*}
    &d'_{\Ts_k}( z(k-1)) = \mnorm{z(k-1) - \proj_{\fix(\Ts_k)}(z(k-1))} \\
    &\quad\leq \mnorm{z(k-1) - \proj_{\fix(\Ts_{k-1})}(z(k-1))} \\
    &\quad+ \mnorm{\proj_{\fix(\Ts_k)}(z(k-1)) - \proj_{\fix(\Ts_{k-1})}(z(k-1))} \\
    &\quad\leq d'_{\Ts_{k-1}}( z(k-1)) + \frac{1}{\sqrt{\pmin}} \varsigma.
\end{align*}
Therefore we now write
\begingroup
\medmuskip=1mu
\thinmuskip=1mu
\thickmuskip=1mu
\begin{align*}
	\ev{d'_{\Ts_k}( z(k))} &\leq \mu \ev{d'_{\Ts_k}( z(k-1))} + \ev{\mnorm{ e_k}} \\
     &\leq \mu \ev{d'_{\Ts_{k-1}}( z(k-1))} + \frac{1}{\sqrt{\pmin}} \left( \mu \varsigma + \ev{\norm{ e_k}} \right).
\end{align*}
\endgroup
Iterating we get
\begingroup
\medmuskip=0mu
\thinmuskip=0mu
\thickmuskip=0mu
\begin{equation*}
\begin{split}
    \ev{d'_{\Ts_k}( z(k))} &\leq \mu^k \ev{d'_{\Ts_0}( z(0))} + \frac{1}{\sqrt{\pmin}} \sum_{h = 1}^{k} \mu^{k - h} \left( \mu \varsigma + \ev{\norm{ e_h}} \right)
\end{split}
\end{equation*}
\endgroup
and using~\eqref{eq:equivalent-norms} again yields the thesis. 
Also, when the problem is static ($\varsigma=0$) there is not any source of error ($e_k=0$ for all $k$), the linear convergence holds also in mean square, indeed, 
\begingroup
\medmuskip=2mu
\thinmuskip=2mu
\thickmuskip=2mu
\begin{equation}\label{eq:conv_ms}
    \ev{d'^2_{\Ts_k}( z(k))} \leq \mu^{2k} \ev{d'^2_{\Ts_0}( z(0))}.
\end{equation}
\endgroup

\textit{[Asymptotic upper bound] - }
We use the same proof technique as in~\cite[Corollary~5.3]{bastianello_stochastic_2022}. Let us define
$$
    y(k) = \max\left\{ 0, d_{\Ts_k}( z(k)) - \sqrt{\frac{\pmax}{\pmin}} \sum_{h = 1}^{k} \mu^{k - h} (\norm{ e_h}+\mu\varsigma) \right\}
$$
for which, by the previous result on the expected distance and Markov's inequality, we have that for any $\varepsilon > 0$ 
$$
    \mathbb{P}[y(k) \geq \varepsilon] \leq \frac{\ev{y(k)}}{\varepsilon} \leq \frac{1}{\varepsilon} \sqrt{\frac{\pmax}{\pmin}} \mu^k d( z(0)),
$$
and, summing over $k$ and using the geometric series,
$$
    \sum_{k = 0}^\infty \mathbb{P}[y(k) \geq \varepsilon] \leq \frac{1}{\varepsilon} \sqrt{\frac{\pmax}{\pmin}} \frac{d( z(0))}{1 - \mu} < \infty.
$$
But by Borel-Cantelli lemma this means that
$
    \limsup_{k \to \infty} y(k) \leq \varepsilon
$
almost surely; and since the inequality holds for any ${\varepsilon > 0}$, the thesis follows.

\section*{Appendix B: Proof of Theorem \ref{thm:loc-msr-static-converr}} \label{app:loc-msr-static-converr}

The first claim is a straightforward consequence of Theorem~\ref{thm:glob-msr} and Lemma~3.1(a) in~\cite{sundharram_distributed_2010}.
For the second claim, notice that the map $\Ts$ is metric subregular at fixed points; indeed, if 
$ z^\star\in\fix(\Ts)$ then $d_{\Ts}( z^\star)=0$ and $\norm{(\Is-\Ts) z^\star} = 0$. This means that $\fix(\Ts) \subset \X$ and, in turn, that $\X$ is a neighborhood of $\fix(\Ts)$ with 
$$
\exists r>0:\quad  \X \supset \{  z \in \rea^m \ | \ d_{\Ts}( z) \leq r \}.
$$
But by the first claim, we know that $ z(k)$ converges almost surely to $\fix(\Ts)$ and, therefore, there exists a finite time $k^\star\in\nat$ after which the $ z(k)$ evolves inside the neighborhood $\X$ in which locally metric subregularity holds.
We can now apply Theorem~\ref{thm:glob-msr} to prove linear convergence in mean for $k\geq k^\star$, completing the proof.

\section*{Appendix C: Proof of Proposition \ref{prop:ULaffMS}}
Given an operator $\Fs:\rea^n\rightarrow\rea^n$, for any $x\in\rea^n$ we denote $\hat{x}^{\Fs}\in \fix(\Ts)$ one of the closest fixed points of $\Fs$ to $x$, namely
$$
\hat{x}^{\Fs} \in \arginf_{ y \in \fix(\Ts)} \norm{ x -  y}\quad \Rightarrow \quad d_{\Ts}( x) = \norm{x-\hat{x}^{\Fs}}.
$$
For each component $i\in=1,\ldots,n$, $\Ls_i$ and $\Us_i$ are affine functions with the same slope but different intercept, yielding
$$
\hat{x}^{\Ls}_i\leq \hat{x}^{\Ts}_i\leq \hat{x}^{\Us}_i \quad \Rightarrow \quad 
\hat{x}^{\Ls}_i-x_i\leq \hat{x}^{\Ts}_i-x_i\leq \hat{x}^{\Us}_i-x_i,
$$
and therefore $\abs{x_i-\hat{x}^{\Ts}_i}\leq \max\{\abs{x_i-\hat{x}^{\Us}_i},\abs{x_i-\hat{x}^{\Ls}_i}\}$.
Thus, the following chain of inequalities holds
\begingroup
\allowdisplaybreaks
\begin{align*}
    d_{\Ts}(x)^2 &= \textstyle\norm{x-\hat{x}^{\Ts}}
   \leq \sum_{i=1}^n \max\{\abs{x_i-\hat{x}^{\Us}_i}^2,\abs{x_i-\hat{x}^{\Ls}_i}^2\}\\
    &\overset{(b)}{\leq} \textstyle \sum_{i=1}^n \max\{ \norm{x - \hat{x}^{\Us}}_\infty^2, \norm{x - \hat{x}^{\Ls}}_\infty^2 \} \\
    &\overset{(c)}{\leq} n \max\{ \norm{x - \hat{x}^{\Us}}^2, \norm{x - \hat{x}^{\Ls}}^2 \}
\end{align*}
\endgroup
where $(b)$ holds since the infinity norm is the maximum distance among each component, $(c)$ holds since $\norm{x}_\infty \leq \norm{x}$. 
We now exploit metric subregularity of $\Us$ and $\Ls$ due to Proposition \ref{prop:ULaffMS} to prove metric subregularity of $\Ts$ as follows
$$
\begin{aligned}
    d_{\Ts}(x) &\leq \sqrt{n} \max\{ \norm{x - \hat{x}^{\Us}}, \norm{x - \hat{x}^{\Ls}} \} \\
    & \leq \sqrt{n} \max\{ \gamma_{\Ls} \norm{(\Is - \Ls)  x}, \gamma_{\Us} \norm{(\Is - \Us)  x} \}\\
    & \leq \sqrt{n} \max\{\gamma_{\Ls} ,\gamma_{\Us} \} \max\{\norm{(\Is - \Ls)  x},\norm{(\Is - \Us)  x} \}\\
    & \leq \gamma_{\Ts}\norm{(\Is - \Ts)  x},
\end{aligned}
$$
where the last inequality always holds for a sufficiently large value of $\gamma_{\Ts}$.
This completes the proof.

\end{document}